\documentclass[a4paper, twoside,10pt]{article}

\usepackage{graphicx,boxedminipage,latexsym}
\usepackage[dvips]{color}
\usepackage[english]{babel}
\usepackage{amsmath,amssymb,amsthm,amscd}
\usepackage[]{fontenc}
\usepackage{amsfonts}
\usepackage{mathrsfs}
\usepackage[mathscr]{eucal} 
\usepackage{xy}
\usepackage{enumerate}

 \allowdisplaybreaks \numberwithin{equation}{section}
\xyoption{all} \numberwithin{equation}{section}

   \usepackage{epsfig}

\begin{document}

\title{A note on fibrations of Campana general type on surfaces}
\author{L. Stoppino}
\date{February 14, 2010}

\newcommand{\av}{``}

\newtheorem{teo}{Theorem}[section]
\newtheorem{prop}[teo]{Proposition}
\newtheorem{lem}[teo]{Lemma}
\newtheorem{conj}[teo]{Conjecture}

\theoremstyle{definition}
\newtheorem{rem}[teo]{Remark}
\newtheorem{defi}[teo]{Definition}
\newtheorem{ex}[teo]{Example}
\newtheorem{ass}[teo]{Assumption}

\renewcommand{\theequation}{\arabic{section}.\arabic{equation}}

\newcommand{\qu}{\mathbb{Q}}
\newcommand{\puno}{\mathbb{P}^1}
\newcommand{\pr}{\mathbb P}
\newcommand{\Z}{\mathbb{Z}}
\newcommand{\R}{\mathbb{R}}
\newcommand{\C}{\mathbb {C}}
\newcommand{\oo}{\mathcal {O}}
\newcommand{\pia}{{\pi_1^{alg}}}
\newcommand{\piu}{{\pi_1}}
\newcommand{\xu}{{X_1}}
\newcommand{\xd}{{X_2}}
\newcommand{\xt}{{X_3}}
\newcommand{\xq}{{X_4}}

\pagestyle{myheadings}
\markboth{\small{Lidia Stoppino }}{\small{\textit{Fibrations of Campana general type on surfaces}}}

\maketitle

\begin{abstract}
We construct examples of simply connected surfaces with genus $2$ fibrations over $\mathbb P^1$ which are of ``general type'' according to the definition of Campana. These fibrations have special fibres such that the minimum of the multiplicities of the components is $\geq 2$ while the g.c.d is $1$. We can extend the construction to any even genus $g$.
\end{abstract}

\section*{Introduction}
In the classification of varieties proposed by Campana  (\cite{Cam}, \cite{camsurvey}), a central r\^ole is played by the concept of general  type 
fibrations.
Roughly speaking, a fibration $f\colon X\longrightarrow Y$ is of general type if either $Y$ is of general type, or $f$ has ``enough
multiple fibres'' (Definition \ref{gtf}).
Campana's definition of multiplicity of a fibre is different from the classical one: in the case of surfaces, 
it is the  minimum of the multiplicities of the components of the fibre, while classical multiplicity is the greater common divisor.
Although the definitions are different, a first natural question is whether there exists or not 
a fibration such that the minimum and the g.c.d of the multiplicities of the components of some fibres
are different; let us call such fibres {\em $C$-fibres}. The problem of existence is solved by a theorem of Winters \cite{Win}  asserting that a curve with locally planar support such that the multiplicities of its components satisfy a natural combinatorial assumption is the fibre of a surjective proper morphism from a smooth surface to a curve.
Winters' result thus reduces the question of the existence of $C$-fibres to a combinatorial one.

It is however interesting to explicitly construct examples of such fibrations, and in particular it is significant from the point of view of Campana's classification to study general type fibrations where the base curve is rational or elliptic.
The arguments used by Winters are of very abstract nature (Deformation Theory and Grothendieck algebrization Theorem), 
and do not allow neither to construct explicit examples nor to  have any control on the base of the fibration, which is obtained by completion from an affine base. 

Campana gives in \cite{Cfm} an example of  general type fibrations of genus 13 over $\puno$ with no classical multiple fibres.
The $C$-fibres in this example are simply connected. This implies that classical multiple fibres and $C$-fibres are different concepts even from the topological point of view; 
indeed, a multiple fibre of a connected fibration can not  be simply connected (see Remark \ref{multversussimply}). 
Moreover, as these are fibrations over $\pr^1$ with a simply connected fibre, the total space  is simply connected itself; this means that fibrations with  $C$-fibres does not have topological restrictions on the fundamental group, as fibrations with classical multiple fibres do (see Remark \ref{multiple}).

It seems interesting to exibit  more $C$-fibres, and to construct  examples in arbitrary genera. 
There are no $C$-fibres in genus $1$; this follows from the classification of singular elliptic fibres done by Horikawa \cite{BHPVdV} V.7.
As the $C$-fibres are very ``complicated'', it seems natural to guess that there is an upper  bound on the genus of fibrations possessing them.
 
This note arises from the observation that in the classification of genus $2$ fibres done by Namikawa-Ueno and Ogg there are exactly four types of $C$-fibres.
Using the local equations provided in \cite{N-U} we construct surfaces with genus two fibrations over $\puno$ and any possible $C$-fibre.
From these, we can form  general type fibrations over $\puno$, and we prove that they are simply connected (Theorem \ref{main}).
It is worth noticing that  genus two fibrations do not admit any  multiple fibre in the classical sense (Remark \ref{nomult}), hence these examples provide
another evidence of the difference between the two notions.
Our construction is very simple, once one is provided with a local equation for these singular fibres; we consider a suitable double cover of $\puno\times\puno$, and apply the canonical resolution process. Eventually, we perform a base change in order to obtain ``enough'' multiple fibres. As an illustration, we develop in detail  in \ref{explicituno} the computations for verifying the existence of the $C$-fibre in one case. Thanks to this insight on the resolution of our surfaces, we  can investigate their geometry and compute their invariants in some cases (Proposition \ref{ratio}).
The fibration we construct are in  two cases isotrivial, in the other ones not (Remark \ref{isotrivial})
From one of the local equations we can eventually derive very naturally a construction for $C$-fibres in any even genus (Theorem \ref{evengenus}).


Let us finally remark that the surfaces here constructed are defined over $\mathbb Q$.
One of the points of view from which Campana's classification is of compelling interest, is the study of diophantine problems. It would be interesting to analyze the density of rational points on these surfaces, in view of  Bombieri-Lang Conjecture
(see Remark \ref{density}).

\medskip

\noindent{\bf Acknowledgements} This note would never have been written without the kind and patient encouragement of Fr\'ed\'eric Campana, to whom I am very grateful. Roberto Pignatelli taught me how to handle special fibres of double covers and gave me countless advices: I thank him heartily. I also would like to thank Margarida Mendes Lopes for her enlightening comments on Winter's result, and Ingrid Bauer for pointing me out a key result on fundamental groups.

\section{Multiple fibres and general type fibrations}\label{C}
We work over the complex field $\C$.
A  \emph{fibred surface}  is the data of a smooth projective surface 
$S$ with a 
surjective morphism with connected fibres (a fibration for short) to a smooth complete curve $B$. 
The general fibres of a fibred surface are smooth and their genus is called the genus of the fibration.
Let  $F=\sum_i m_iC_i$ be a fibre of a fibration, where the $C_i$'s are the irreducible components. The fibre $F$ is said to be {\em multiple} of multiplicity $n$ if $n=g.c.d.\{ m_i\}>1$. 


\begin{rem}\label{multiple}\upshape{
Given a fibration $f\colon S\longrightarrow B$ with multiple fibres,
it is - almost every-time - possible to construct an \'etale cover of $S$ by an appropriate base change.
Indeed, from \cite{namba} we know that for any smooth curve of positive genus $B$, given a set of points 
$\{p_1,\cdots p_k\}$, and a string of integers strictly greater than $1$, $(m_1,\cdots,m_k)$, 
there exists a cover $B'$ of $B$, ramified exactly above the $p_i$'s, with ramification index $m_i$ on $p_i$. 
Of course this cover in general is not cyclic.
For $B=\puno$ the above statement holds except for the cases $k=1$, and $k=2$, $m_1\not= m_2$.
Given a fibration $f\colon S\longrightarrow B$ with multiple fibres 
(with the above restrictions if $B=\puno$),
consider the commutative diagram 
$$
\begin{CD}
S' @>\alpha>>  S\\
@VVV @VVfV\\
B'@>\psi>>B
\end{CD}
$$

\medskip

\noindent where  $\psi\colon B'\longrightarrow B$ is the  the cover ramified exactly at the points
corresponding to multiple fibres, with ramification index equal to the multiplicity, and 
 $S'$ is the normalization of the fibre product.
Then the morphism $\alpha $ is \'etale (cf. \cite{CamN}, Appendix A.C).}
Note that the fundamental group of $S$ contains a finite index subgroup $H$ which admits a surjective homomorphism on the fundamental group of a curve of genus $\geq 2$.
\end{rem}

Let $f\colon S\longrightarrow B$ be a fibred surface. 
For any fibre $f^*(b)=\sum_i m_iC_i$, we define $m(b)$ as $\min_i\{m_i\}$.
As the general fibre of $f$ is smooth, $m(b)=1$ for all but a finite set of points in $B$.
Consider  the $\qu$-divisor  on $B$
$$\Delta (f) =\sum_{b\in B} \left( 1-\frac{1}{m(b)}\right) b.\footnote{The data $(B, \Delta(f))$ is called the {\em base orbifold}, or the {\em constellation} \cite{Abr} of $f$.}$$
Let $\kappa (B, K_B+\Delta (f))$ be 
the Kodaira dimension of the $\qu$-divisor $K_B+\Delta (f)$, i.e. the Kodaira dimension of an integer multiple of  
$K_B+\Delta (f)$.

\begin{defi}[Campana]\label{gtf}
The fibration $f\colon S\longrightarrow B$ is said to be of general type if $$\kappa (B, K_B+\Delta (f))=2.$$
\end{defi}

\begin{ex}\label{multpuno}
\upshape{
Clearly, any fibration over a curve of genus greater or equal to two is of general type. A fibration over an elliptic curve $B$ is of general type if there is at least a $b\in B$ such that $m(b)\geq 2$.
A fibration  $f$ over $\puno$ is of general type if and only if 
$$-2+\sum_b \frac{m(b)-1}{m(b)}=deg (K_B+\Delta (f))>0.$$
Hence, for instance, no fibrations over $\puno$  with exactly two singular fibres are of general type.
If we consider the  string of integers given by the $m(b)$'s which are not equal to $1$, listed increasingly,
it is easy to see that  the cases in which $f$ over $\puno$ is not of general type are associated to the following strings:
$$(n), (n,m), (2,2,n), (2,3,k), (2,4,4), (3,3,3), (2,2,2,2),$$
with $n,m$ arbitrary integers strictly greater than $1$, and $2\leq k\leq 6$.
 Note that these are precisely the cases when the associated ramified cover of $\puno$
- as defined in Remark \ref{multiple} - has genus strictly smaller than $2$.}
\end{ex}

From Remark \ref{multiple} and Example \ref{multpuno} above, it is clear that if $f\colon S\longrightarrow B$ is a general type fibration  such that for any $b\in B$ the integer $m(b)$ equals  the classical multiplicity, then $S$ is of general type, because it has an \'etale cover of general type.
More generally, in \cite{Cfm}, Proposition 1.7 (see also \cite{Cam} Section 3.5) it is shown that if  $f\colon S\longrightarrow B$ is a general type fibration, then $S$ has Kodaira dimension $2$ if and only if the genus of $f$ is at least $2$.

\begin{defi}
Let $f\colon S\longrightarrow B$ be a fibred surface. We will say that a fibre $f^*(b)$ is of Campana type, and call it $C$-fibre, 
if it is not a multiple fibre, but $m(b)>1$.
\end{defi}

We now turn to the question of the existence of general type fibrations with $C$-fibres.
The following  result of Winters \cite{Win} holds. 

\begin{teo}[Winters]\label{winters}
Let $Z=\sum_iC_i$ be any reduced locally planar curve.
Let $\{m_i\}_{i\in I}$ be  a collection of positive integers.
If for any $i$, $m_i$ divides  $\sum_{j\not= i}m_j(C_i\cdot C_j)$, then there exists a proper surjective morphism from a smooth surface to a smooth complete curve such that $C=\sum_i m_iC_i$ as a fibre of such a morphism. 
\end{teo}

\begin{rem}\label{multversussimply}One has to be aware that Winters' result does not establish the existence of a fibration, i.e. of a proper morphism with connected fibres. Note for instance that if $G$ is any simply connected locally planar reduced curve then for any $m$ the curve $C=mG$ satisfies the conditions of Theorem \ref{winters} and hence it can be seen as the fibre of a morphism $f\colon S\longrightarrow B$. However, necessarily $f$ has not connected fibres. Indeed if $f$ would be a fibration, the sheaf $\mathcal O_{C|G}$ would be a torsion sheaf on $G$ (\cite{BHPVdV}, Lemma III.8.3), so that $\pi_1(G)\not = \{1\}$, a contradiction.
\end{rem} 

On the other hand, it is immediate to see that the case of a classically multiple fibre is the only case when the morphism of Winters is not a fibration: 
\begin{prop}
Let $C=\sum_i m_iC_i$ be the fibre of a proper surjective morphism $f\colon S\longrightarrow B$ from a smooth projective surface $S$ to a smooth complete curve $B$. Suppose that $C$ is connected and  that g.c.d.$\{ m_i\}=1$. Then $f$ has connected fibres; hence, it is a fibration.
\begin{proof}
Consider the Stein factorization of $f$
\begin{equation}\label{diag}
\xymatrix{
\widetilde S \ar[dr]_{f'}\ar[rr]^f &&B \\
&B'\ar[ur]_\gamma\\}
\end{equation}
Where $\gamma$ is a finite morphism of degree $d$ and $f'$ a fibration (in particular it has connected fibres). Suppose that $f$ has not connected fibres: equivalently $d\geq 2$. Let $C$ be as above, and $b\in B$ such that $f^*(b)=C$. 
As $f^*=f'^*\circ \gamma^*$, the support of $f^*$ is connected if and only if $b$ is a branch point of total ramification for $\gamma$. But then $C=f'^*(\gamma^*(b))=f'^*(db')=dF'$, hence g.c.d.$\{ m_i\}\geq d\geq 2$, contrary to the assumption.
\end{proof}
\end{prop}

As pointed out in the Introduction, Winters' result is not constructive, and in particular it does not guarantee the existence of fibrations with $C$-fibres and no multiple fibres, and the existence of general type fibrations over base curves of genus smaller than one. 

The first example of fibration with $C$-fibres  is provided by Campana in \cite{Cfm}, Sec. 5.
The author constructs a genus $13$ fibration over $\puno$ without multiple fibres and with a fibre made of $5$ rational smooth curves all meeting in a point, 3 with multiplicity 2 and 2 with multiplicity 3, as in figure below.
\smallskip
\begin{center}
{}{}\scalebox{0.30}{\includegraphics{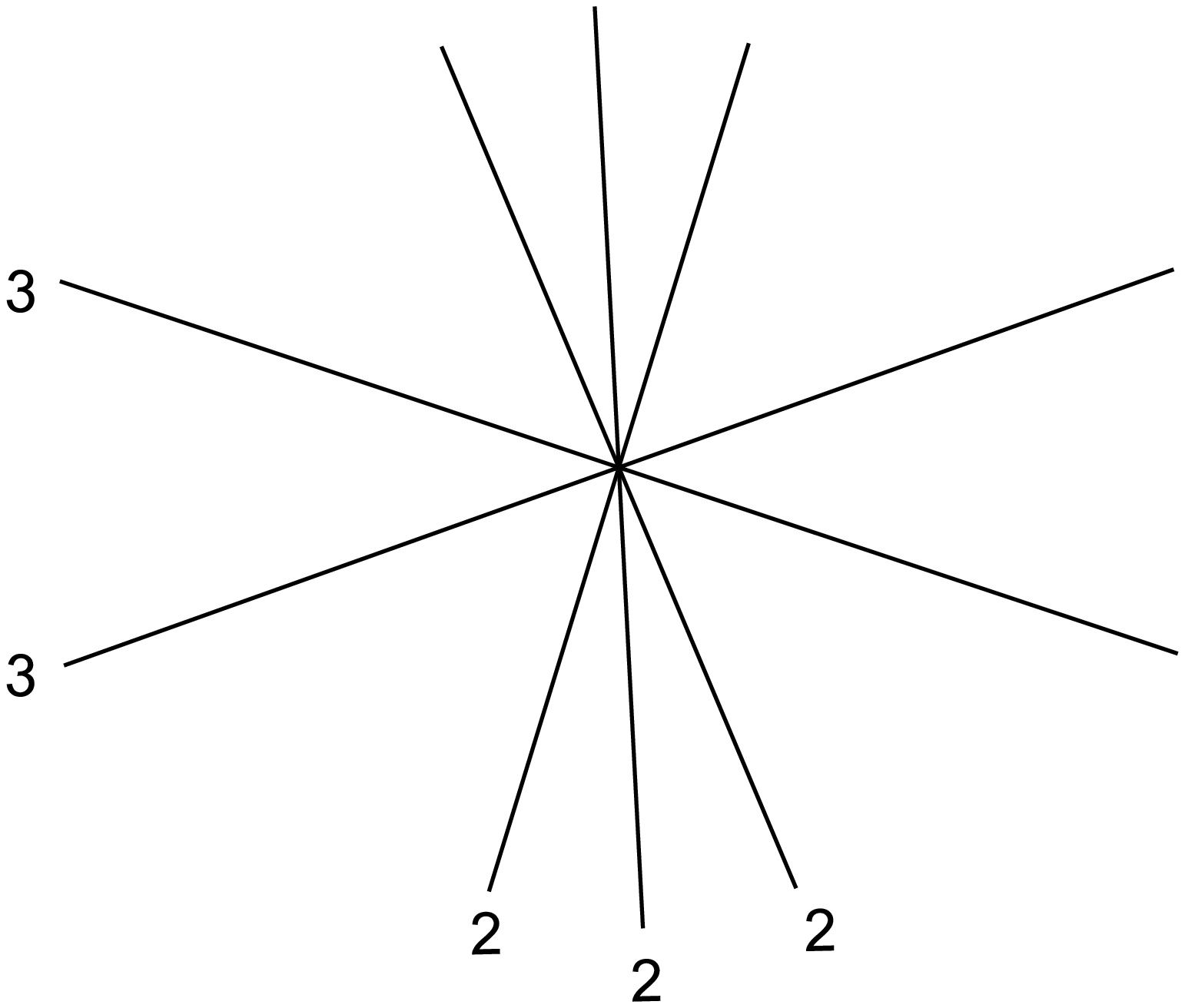}}
\end{center}

From this, by performing a suitable base change, he can produce a surface fibred over $\puno$ with at least $6$ fibres of the above type.
Moreover, the surface is simply connected, because it has a fibration over $\puno$ with a simply connected fibre.


\section{Campana's fibrations in genus two}\label{C2}

The fact that  $C$-fibres don't exist in genus one follows from the classification of singular elliptic fibres 
done by Horikawa (\cite{BHPVdV} V.7). 
It seems natural to think that in low genus no $C$-fibres exist.
For instance, it is not hard to see that  the lowest possible genus of a $C$-fibre made of  two irreducible components is $11$,  achieved 
by a curve of the form $F=2A+3B$ where $A$, $B$ are two smooth rational curves meeting transversally in $6$ points.
However, from the lists given by Ogg \cite{ogg2} and by Namikawa-Ueno  \cite{N-U}, we see that there are $C$-fibres already in genus two. There are in fact  $4$  possible configurations, having much more than two irreducible components; listed in Figure 1 below.
For the reader's convenience, we report in the figure the corresponding numbers of Ogg's list. The straight lines in the picture are smooth rational components, the number attached to them are, as usual, their multiplicities.
\begin{center}

{}{}\scalebox{0.60}{\includegraphics{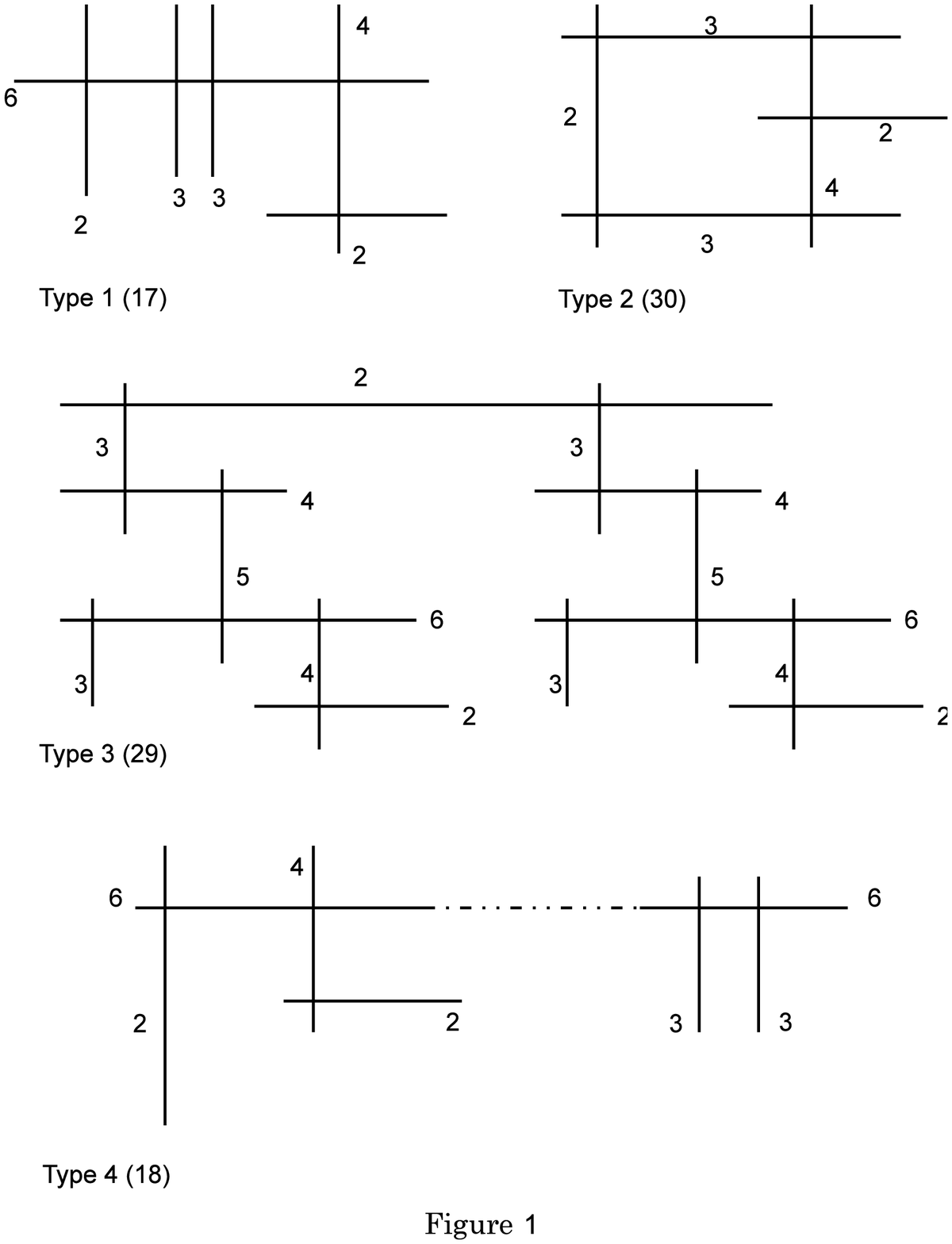}}
\end{center}
\vspace{-1.5cm}

\begin{rem}\label{nomult}
The fact of finding $C$-fibres of genus two fibrations is another evidence of the difference between this kind of fibres 
and the multiple ones, because   genus two fibrations does not admit multiple fibres.
Indeed, this follows from this more general simple remark.
Let $S\longrightarrow B$ be a fibration of genus $g$. 
Then, for any fibre $F$, by the adjunction formula, $(K_S\cdot F)=2g-2$. 
Suppose that $F$ is a multiple fibre, i.e. $F=nE$, with $E$ effective.
Then $(K_S\cdot E)= (2g-2)/n.$
The arithmetic genus of $E$ is 
$$p_a(E)= 1+\frac{1}{2}(K_S\cdot E)=1 + \frac{g-1}{n}.$$
Hence, given any fibration of genus $g$, the multiplicity of the fibres divides $g-1$. 
\end{rem}
The main result can be summarized as follows (Proposition \ref{prop1} and Proposition \ref{cornalba} below).

\begin{teo}\label{main}
For any $i=1,\ldots,4$ there exist general type fibrations $g_i\colon S_i\longrightarrow \pr^1$ such that \\
$(i)$ the surfaces $S_i$ are of general type and simply connected;\\
$(ii)$ Any $g_i$ admits fibres of type $i$, and no multiple fibres.
\end{teo}

We construct the surfaces by using the explicit equations for the local deformations of such fibres given  by Namikawa and Ueno.
Indeed the two authors give for any fibre a local equation of a family of sextics which realize locally a smooth family with the desired fibre (some of these equations were also found by Ogg):
$$
\begin{tabular}{|c|c|}
\hline
type $1$& $y^2=t(x^6+\alpha tx^3+t^2)$\\
\hline
type $2$ & $y^2=t(x^2+t)(x^4+t)$\\
\hline
type $3$ & $y^2=t(x^3+t^2)((x-1)^3+t^2)$\\
\hline
type $4$ & $y^2=t((x^3-t)^2+t^{h+2})$\\
\hline
\end{tabular}
$$

\medskip

The coordinate $t$ is the local coordinate of the base, $x$ of the fibre.
When one makes a canonical resolution (\cite{BHPVdV}, III.7 and \cite{P}) of this double cover 
the result is a smooth surface with a fibration over a complex disc such that the fibre over $0$ is the desired one. As an illustration, we shall develop in detail this computation for the first case in \ref{explicituno}.

\medskip

Let us call $([x:z],[t:s])$ the coordinates in $\pr^1\times \pr^1$. We projectivize the four  equations of the branches as follows
$$
(1)\,\, ts(s^2x^6+\alpha stx^3xz^3+t^2z^6)\quad(2)\,\,ts(sx^2+tz^2)(sx^4+tz^4)
$$
$$
\quad(3)\,\, ts(t-s)(s^2x^3+t^2z^3)(s^2(x-z)^3+t^2z^3)\quad 
(4)\,\, ts(s^{h}(sx^3-tz^3)^2+t^{h+2})
$$
Let $B_i$ be the curve defined by equation $(i)$. These are even divisors in $\puno\times\puno$, i.e. there exist a line bundle $\mathcal L_i$ such that $\mathcal L_i^{\otimes 2}\cong \mathcal O_{\puno\times\puno}(B_i)$.
Let $Y_i$ be the double cover of $\puno\times\puno$ ramified over these branches. The $B_i$'s are singular reduced curves, so $Y_i$  are singular normal surfaces \cite{P}. Apply the canonical resolution process to $Y_i\longrightarrow \puno\times\puno$, and let  $X_i$ be the minimal model of the surface obtained. 
Any of the $X_i$'s admits a fibration of genus two over $\puno$ with the fibre  over $[0:1]$ of  type $i$, and possibly some other  singular fibres, none of which is multiple by Remark \ref{nomult}. In particular, by the symmetry of equations $(1)$ and $(2)$ we see that $X_1$ and $X_2$ have another fibre of type $1$ and $2$ over $[1:0]$.

In order to produce a fibration of general type from our construction, we argue as Campana in \cite{Cfm}. 
Consider $n\in \mathbb N_{\geq 2}$, and perform a cyclic  base change of degree $n$ from $\puno $ to $\puno$ ramifying in two points other than $[0:1]$ and $[1:0]$. Let $X_i^n$ be the normalization of the induced fibred product.
This way, we obtain a fibration $f^n_i\colon X_i^n\longrightarrow \pr^1$ with at least $2n$  fibres of type $i$ for $i=1,2$, or at least $n$  fibres of type $i$ for $i=3,4$.
To produce a fibration over an elliptic curve, we  can for instance make a base change of degree $2$ (resp. $3$) 
ramifying on $4$ (resp. $3$) points other than $[0:1]$. 
For $n\geq 6$  these fibrations are of general type (Definition \ref{gtf}) and  - by  \cite{Cfm}, Prop.1.7 - the surface obtained are of general type. 
We have thus proved the following result.

\begin{prop}\label{prop1}
For $i=1,2$ and $n\geq 3$ (resp. $i=3,4$ and $n\geq 6$) the fibrations $f^n_i\colon X_i^n\longrightarrow \pr^1$ are of general type with fibres of type $i$, and no multiple fibres. The surfaces $X_i^n$ are of general type.
\end{prop}

We now turn our attention to the fundamental groups of our surfaces. For $i=1,3,4$, the surfaces $X_i^n$ are simply connected, having a simply connected fibre and base curve (\cite{Cfm}, Lemme 5.8). 
This result can not be applied to the surfaces $X_2^n$.
However, we can prove directly that all our surfaces are indeed simply connected using a Lefschetz-type result of Cornalba  \cite{cornish} whose statement is the following. If $X$ is a smooth variety and $X'\longrightarrow X$ a cyclic covering whose branch locus is ample, then the fundamental groups of $X$  and $X'$ are isomorphic.
\begin{prop}\label{cornalba}
Let $X_i\longrightarrow \pr^1$ as above. Let $\beta \colon B\longrightarrow \pr^1$ any covering and let $X_i^\beta$ the desingularization of the fibred product induced by $\beta$. Then $\pi_1(X_i^\beta)\cong \pi_1(B)$. In particular, the surfaces $X_i^n$ are simply connected for any $i=1,\ldots ,4$.
\begin{proof}
Let $n$ be the degree of $\beta$. Let $Y_i^\beta$ be the fibred product induced by $f_i\colon X_i\longrightarrow \pr^1$ and $\beta$, and let $B^\beta_i$ be the branch locus of the double cover   $Y_i^\beta\longrightarrow B\times \pr^1$.
Observe that $B^\beta_i$ is of bi-degree $(4n, 6)$ for $i=1,2$ (resp. $(6n, 6)$ for $i=3$, $(n(h+4), 6)$ for $i=4$) in $B\times \pr^1$. Hence (\cite{Har} sec.V.2) $B^\beta_i$ is ample for any $i$, and we can apply Cornalba's result, obtaining that $\pi_1(Y_i^\beta)\cong \pi_1 (B)$. Now observe that $X_i^\beta$ is a desingularization of $Y_i^\beta$, and hence their fundamental groups are isomorphic.
\end{proof}
\end{prop}

\begin{rem}\label{density}
It is worth noticing that these surfaces, together the fibrations, are defined over $\mathbb Q$. This could be interesting from the point of view of arithmetic geometry, in view of the celebrated Bombieri-Lang conjecture: any surface of general type defined over a number field has rational points not  potentially dense (see \cite{Abr} for a general discussion, and for references on this topic). Also from this point of view there is a dichotomy between multiple and $C$-fibres. Indeed, a general type fibration without $C$-fibres is easily seen not to be potentially dense, using the induced \'etale covering from a surface fibred over a curve of genus $\geq 2$, as in \cite{CT/S/SD}. 
For varieties defined over number fields, Campana conjectures that the ones with potentially dense rational points are exactly those not admitting any general type fibration. He proves in \cite{Cfm} the hyperbolic and function field version of this statement, while the arithmetic one remains open, depending on an ``orbifold'' version of the Mordell conjecture.
It could be interesting to check the potential non-density of rational points in our explicit cases.
\end{rem}

\section{Explicit computations and the even genus case}\label{explicit}
\subsection{Canonical resolution for the first case}\label{explicituno}
We now write down explicitly the canonical desingularization in the first case, and verify by hand that there are two fibres of type 1.

Let $\pi\colon X\longrightarrow W$ be a double cover with $W$ smooth and  branch divisor $R$ singular and reduced.
  the  canonical resolution
(see \cite{BHPVdV} III.7) is a standard construction that allows to obtain a smooth surface birational to $X$. Note that this surface is not necessarily minimal.
This process consists of resolving by successive blow ups the singularities of $R$, and forming step by step an associated double cover, 
producing  the following commutative diagram.
$$
\begin{array}{cccccccccl}
 X_k & \stackrel{\sigma_k}{\longrightarrow} & X_{k-1} & \longrightarrow &...& \longrightarrow & X_1 & \stackrel {\sigma_1}{\longrightarrow} & X_0 & =X \\
\downarrow &  & \downarrow & & & &\downarrow & & \downarrow &\\
Y_k & \stackrel{\tau_k}{\longrightarrow} & Y_{k-1}& \longrightarrow & ... & \longrightarrow & Y_1 &\stackrel{\tau_1}{\longrightarrow} & Y_0 &=Y\\
\end{array}
$$ 
\noindent The $\tau_j$ are successive blow-ups that resolve the singularities of $R$; 
and the morphism $X_j\rightarrow Y_j$
is the double cover with branch locus 
\begin{equation}\label{ramificazione}
R_j:=\tau_j^*R_{j-1}-2\left[\frac{m_{j-1}}{2}\right]E_j,
\end{equation}
where $E_j$ is the exceptional 
divisor of $\tau_j$, $m_{j-1}$ is the multiplicity of the blown-up point, and $[\ \ ]$ stands for integral part.
At the end of this process, we obtain a smooth surface birational to the original one. 

Let us consider the first equation.
Recall that the branch divisor is (taking $\alpha=1$)

\medskip
\centerline{$R:=B_1=Z(st(x^6s^2+s tx^3z^3+t^2z^6))\subset \puno\times \puno,$}

\smallskip
\noindent where $([t\colon s], [x\colon z]) $ are the coordinates of $\puno\times \puno$.
We now apply the canonical resolution machinery to the singular point $p=([0\colon 1], [0\colon 1])$ of $R$.
Let us follow the the local behavior of the branch loci under process of canonical resolution.

Let us consider the affine subset given by the equation $t(x^6+ tx^3+t^2)=0$ in the affine space $\{s=t=1\}\cong \mathbb C^2$, and 
Let $C_1=\{t=0\}$, $C_2=\{2x^3+t(1+\sqrt{t^2-4})=0\}$ and $C_3=\{2x^3+t(1-\sqrt{t^2-4})=0\}$ be the local components of the branch divisor near $p$.
Let $\tau_1$ be the blow up in $p=(0,0)$; let $E_1$ be the exceptional divisor of $\tau_1$. As $p$ has multiplicity $3$ we get as total transform 
$$
\tau_1^*(C_1+ C_2 +C_3)= \widetilde C_1+ \widetilde C_2 + \widetilde C_3+ 3E_1, 
$$
where the $\widetilde C_i$'s are the strict transforms of the $C_i$'s.
A local affine equation for the total transform is $ux^3(x^4+ux^2+u^2)=0$.
Hence the branch divisor of the first double cover, by equation (\ref{ramificazione}),  is
$$R_1:=\tau_1^*(C_1+ C_2 +C_3)-2E_1= \widetilde C_1+ \widetilde C_2 +\widetilde C_3+ E_1.$$

By local computations as above it is immediate to see that this branch locus has a singular point $p_1$ and the equation around this point is $ux(x^4+ux^2+u^2)=0$. So, this point has multiplicity $4$.
Let $\tau_2$ be the blow up of $p_1$, the branch divisor now becomes (with the obvious notations) 
$$R_2:=\tau_2^*(\widetilde C_1+ \widetilde C_2 +\widetilde C_3+ E_1)-4E_2=\widetilde {\widetilde C_1}+ \widetilde{\widetilde C}_2 +\widetilde{\widetilde C}_3+ \widetilde E_1.$$
Again by local computations, we see that $E_2$ separates $E_1$ and the $C_i$'s.
The process goes on until the branch locus is (disconnected and) smooth as illustrated in the figure below. 

\vspace{0.2cm}
\begin{center}
{}{}\scalebox{0.40}{\includegraphics{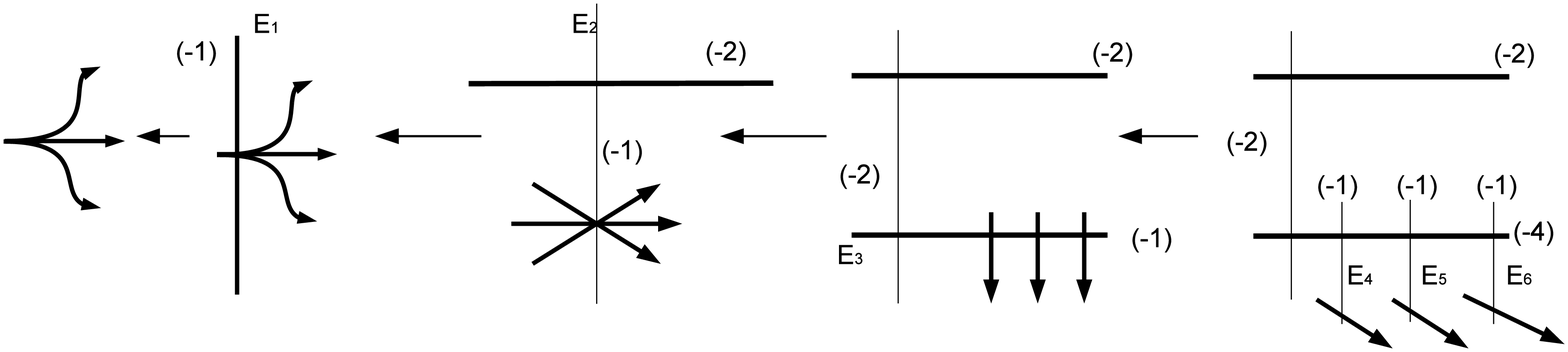}}
\end{center}
\vspace{0.2cm}
\noindent The bold lines  indicate the branch locus $R_j$ at every step of the desingularisation.
The curves with arrows  indicates the strict transform of the local branches of the curve $R$ near $p$. 
The straight lines are the exceptional divisors arising from the blow ups.
The numbers in brackets are the self-intersection of the exceptional divisors at each step.
In the last step, we blow up all the three singular points of the branch locus.
If we consider the  double cover branched along the last branch locus, the pullback of this  divisor is as follows.

\begin{center}
{}{}\scalebox{0.40}{\includegraphics{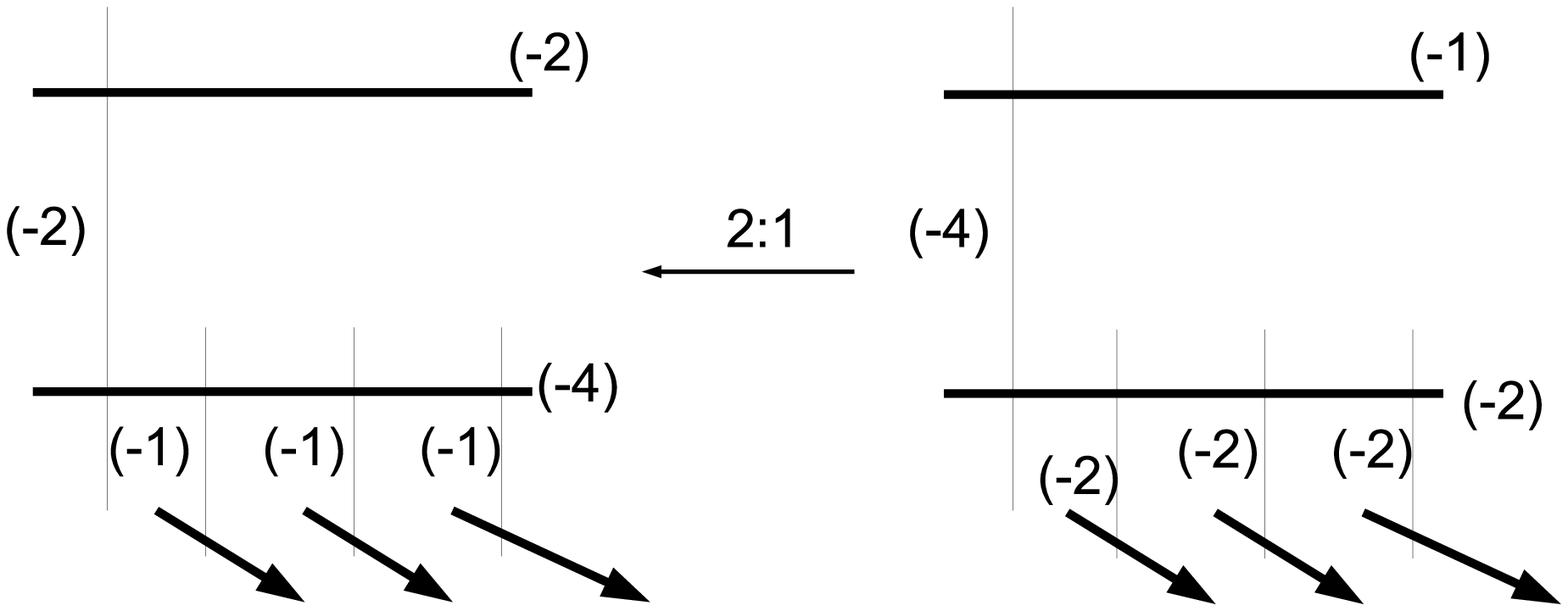}}
\end{center}

\noindent Observe that a $-1$-curve arises after the double cover: we therefore contract it.

The branch locus $R$ has only one other singular point, namely $q=([1\colon 0], [1\colon 0])$, and the local 
equation near $q$ is exactly the same as the one near $p$. 
Let $Z_1$ be the smooth surface obtained applying the above process to both $p$ and $q$.
Let us consider the fibration $f_1\colon Z_1\longrightarrow \puno$ obtained from the projection of $\puno\times\puno$ 
on the first coordinate. 
In order to understand the fibre over $t=0$, we have to ``follow''  $t=0$ in the resolution, as showed by the figure below. 
\vspace{0,1cm}

\begin{center}
{}{}\scalebox{0.40}{\includegraphics{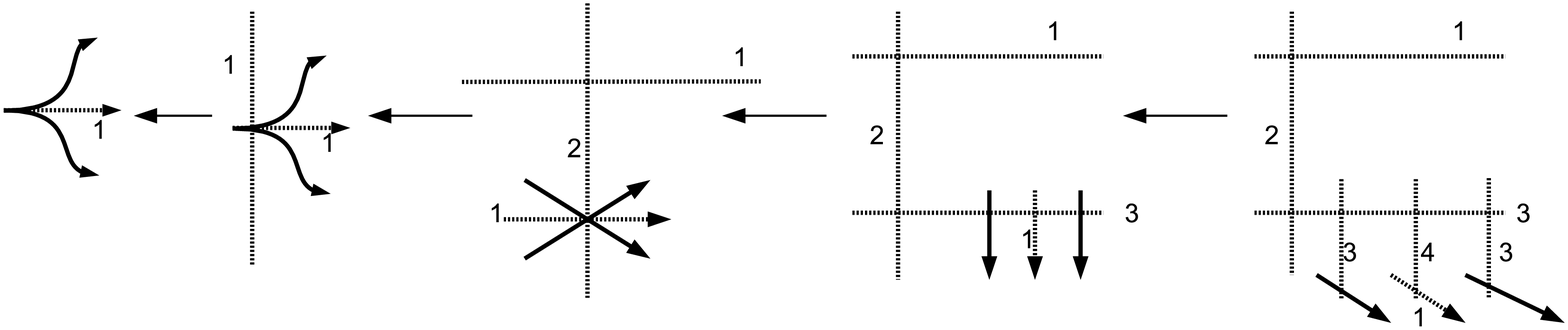}}
\end{center}
\noindent  The dotted line is the fibre on each step, and as usual the numbers in the figure represent the multiplicities of the components.

Applying now the double cover - remembering that the multiplicity of the components contained in the ramification locus get multiplied by $2$, while the other multiplicities remain unchanged - we see that the fibre becomes exactly of the desired type.

\begin{rem}\upshape{
Note that $Z_1$  also has an elliptic  fibration  (i.e. a fibred surface such that the general fibre is an elliptic curve) 
$e_1\colon Z_1 \longrightarrow \puno$, obtained considering the second projection.
Arguing as we did for the genus two fibration, we can see that $e_1$ has only two singular fibres of the form
\vspace{0.2cm}
\begin{center}
{}{}\scalebox{0.48}{\includegraphics{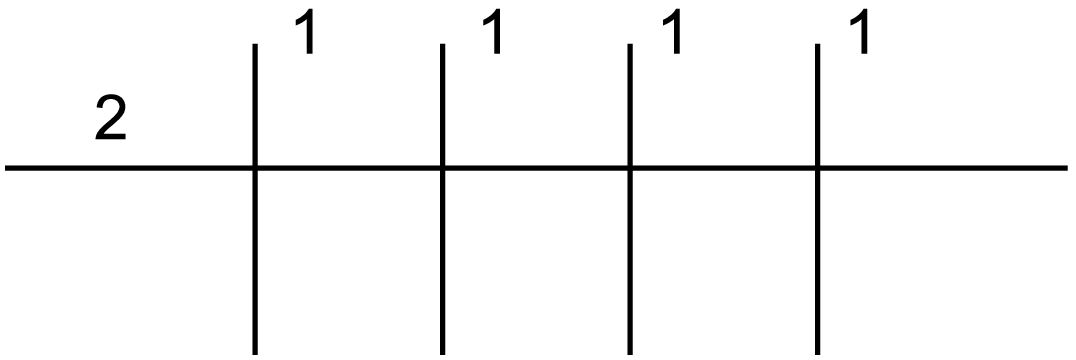}}
\end{center}
}
\end{rem}

\subsection{Invariants of $X_1^n$}
We can now compute explicitly  the invariants of the surface obtained via the first  equation. Let us first recall from the first section \cite{P} that an {\em infinitely closed triple point of order $k$} is a triple point of a locally planar reduced curve such that it turns into an ordinary triple point after $k$ successive blow ups.
\begin{prop}\label{ratio}
Let   $n\geq 2$, and let $X_1^n$ be as in Section \ref{explicit}.
This surfaces are of general type and their  Chern invariants are 
$$c_1^2(X_1^n)=6n-8, \quad c_2(X_1^n)=18n-4.$$
\end{prop}
\begin{proof} 
Let us start by considering the branch locus $B_1=Z(st(x^6s^2+s tx^3z^3+t^2z^6))$. From \ref{explicituno}, we have that $B_1$ has only two singular points $p=[0:1]$ and $q=[1:0]$, that both are infinitely close triple points of order $2$.

Let $\beta\colon \puno\longrightarrow \puno$ be a cyclic base change of degree $n$ that ramifies on two points other than $p$ and $q$.
As in Proposition \ref{cornalba}, let us call $Y_1^n$ the normalization of the fibred product induced by $\beta$ and $f_1$.
Let us call $Z_1^n$ the smooth surface obtained from $Y_1^n$ applying the canonical resolution, as described above, to the $2n$ singular points of the branch and then contracting all the $-1$-curves in the fibres. 
The surface $X_1^n$ is the minimal model of  $Z_1^n$, but we are going to show that  $Z_1^n$  is itself minimal, so that the two surfaces coincide.
Let $H_1$ and $H_2$ be the generators of the N\'eron-Severi group of $\puno\times \puno$. By the formulas for the invariants of double covers we have that
$$c_1^2(Y_1^n)=2((2n-2)H_1+H_2)^{\cdot 2}=8n-8,$$
and 
$$
\begin{array}{ll}
c_2(Y_1^n)&=2c_2(\puno\times \puno)+ (K_{\puno\times \puno}+4nH_1+6H_2\cdot 4nH_1+6H_2)=\\
 &= 8+16n+6(4n-2)=40n-4.\\
 \end{array}$$
Hence, $\chi(Y_1^n)=1/12(c_1^2(Y_1^n)+c_2)=4n-1$.
As all the $2n$ singular points of the branch are infinitely close triple points of order $2$, by the generalization of Prop. 1.11 in \cite{P}, we have that 
$$c_1^2(Z_1^n)=c_1^2(Y_1^n)-2n=6n-8,$$
$$\chi(Z_1^n)=\chi(Y_1^n)-2n=2n-1,$$
and hence $c_2(Z_1^n)=18n-4$.
To prove that $Z_1^n$ is minimal, we can observe that any $-1$-curve inside this surface is either the pullback of an exceptional not meeting the branch locus divisor in the blow up of $\puno\times \puno$, or the reduced pullback of a $-2$-curve inside the branch locus. But we see from the explicit computations in \ref{explicituno} that there are no curves of the first type, whereas there are $2n$ curves of the second kind are already contracted in $Z^n_1$. Note that we could also use directly Corollary of Prop. 4.2 of Xiao in \cite{XLNM} to conclude that $Z_1^n$ is minimal.

The fact that $X_1^n$ is of general type follows from the Enriques-Kodaira classification of surfaces.
\end{proof}
Note that we prove directly, without using Campana's result \cite{Cfm}, that these surfaces are of general type for $n\geq 3$, which
is the case in which the fibration is of Campana general type.
Moreover, we see that also the case $n=2$ gives a general type surface, although the fibration is not of Campana general type.

\begin{rem}\label{isotrivial}
Recall that a fibration is \emph{isotrivial} if all its smooth fibres are mutually isomorphic. 
It is easy to see that the $f_1^n$'s are isotrivial fibrations.
For instance it can be checked directly that  the stable reduction of the singular fibres of this fibrations is  a smooth curve of genus two\footnote{This is another evidence of the somehow surprising fact that the semistable reduction of of even extremely complicated 
non-semistable curves could be very simple; in this case even smooth, cf. \cite{H-M}, sec.3.C.}.

Differently from $f_1$, the fibration $f_2$ has another  singular fibre over $r=[1:1]$.
Moreover, the stable reduction of the two $C$-fibres is a curve of geometric genus $1$ with one node. 
Therefore, $f_2$ is an example of a non-isotrivial fibration over $\puno$ having  the smallest possible number of singular fibres, 
according to \cite{Beaumin}.
Note that $X_2$, as the examples given in \cite{Beaumin}, is a rational surface. By making a base change of degree $2$ that ramifies in two of the points corresponding to singular fibres, we obtain a fibred surface of general type with $4$ singular fibres. If, on the other side, we make a base change in the direction of the elliptic fibration, it is easy to check that we obtain a surface with Kodaira dimension $1$ with a non-isotrivial fibration of genus $2$ having three singular fibres. 

For non-isotrivial semistable fibrations, on the contrary, it holds that the number of singular fibres is greater or equal to four \cite{Beaumin}, and if it is less or equal to $5$, then the Kodaira dimension of the surface is negative \cite{Tanmin2}.

As for the other fibrations, it is a simple check to see that $f_3$ is isotrivial, while $f_4$ is not.
\end{rem}


\subsection{Even genus fibrations of general type}
With a simple modification of the third equation we can produce other fibrations with $C$-fibres, in arbitrary even genus.
\begin{prop}\label{echecacchio}
Let $n\geq 2$ be an even integer and let $\omega\in \mathbb C$ be a primitive $n$-th root of unity. 
Let $$R^n:=Z(st \prod_{i=1}^n(s^2(x-\omega^iz)^3+t^2z^3))\subset \puno\times\puno.$$

Let $Z^n$ be the canonical resolution of the double cover of $\puno\times \puno$ ramified over $R^n$, and let $f^n\colon Z^n\longrightarrow \puno$ be the fibration induced by the projection on the first factor. Then $f^n$ is of genus $n$ with fibre over $[0:1]$ of the form described in figure below.
\begin{center}
{}{}\scalebox{0.60}{\includegraphics{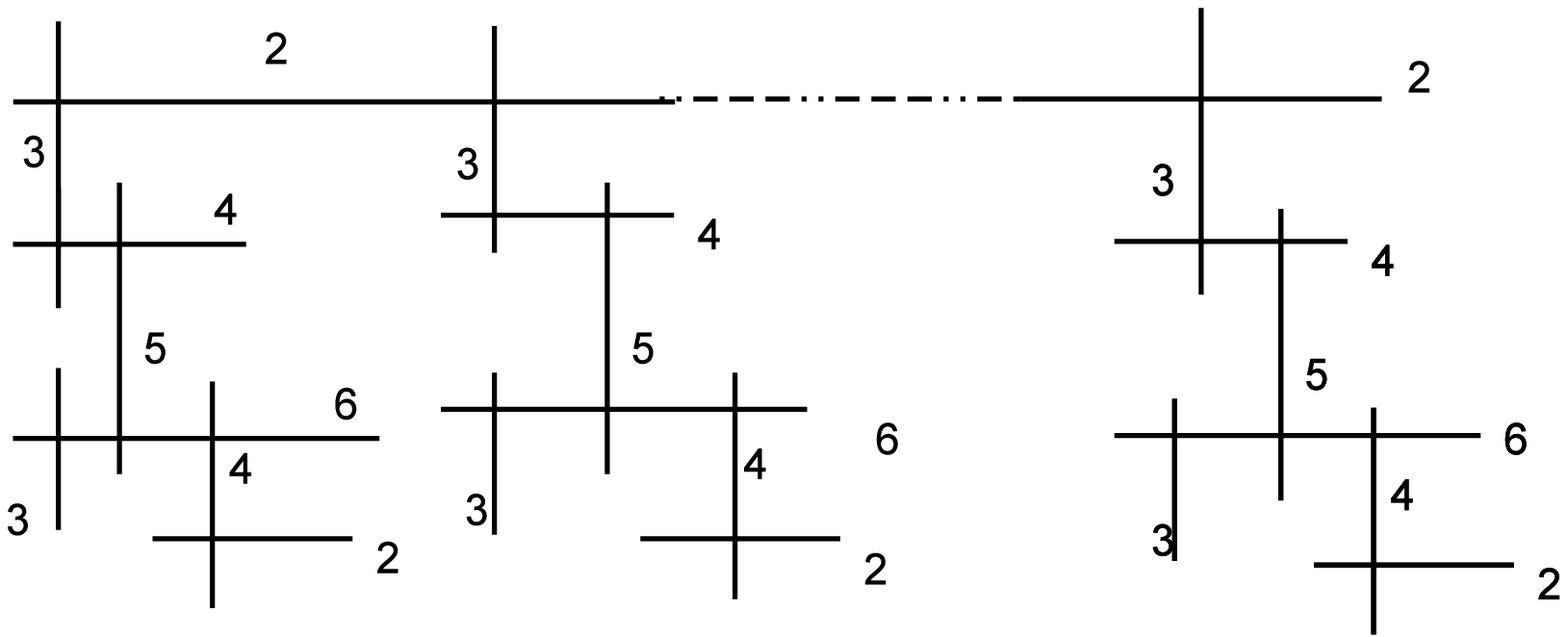}}
\end{center}
\vspace{0,2cm}
\end{prop}
\begin{proof}
The branch $R^n$ near the line $\ell=\{ t=0\}$ has as local components $\ell$ and $n$ branches which have a cusp at the points $(0, \omega^i)$, for $i=1,\ldots n$. It is immediate to check, via similar computations as in \ref{explicituno}, that these cusps produce in the fibre of $f^n$ over  $[0:1]$ $n$ tails made  of rational multiple components as in the figure, attached to a rational component with multiplicity $2$, which is the pullback of the strict transform of $\ell$. This process can be better illustrated via the following figures. As usual the bold lines are the components of the branch locus. The dotted line represents the strict transform of $\ell$.
\vspace{0.7cm}
\begin{center}
{}{}\scalebox{0.84}{\includegraphics{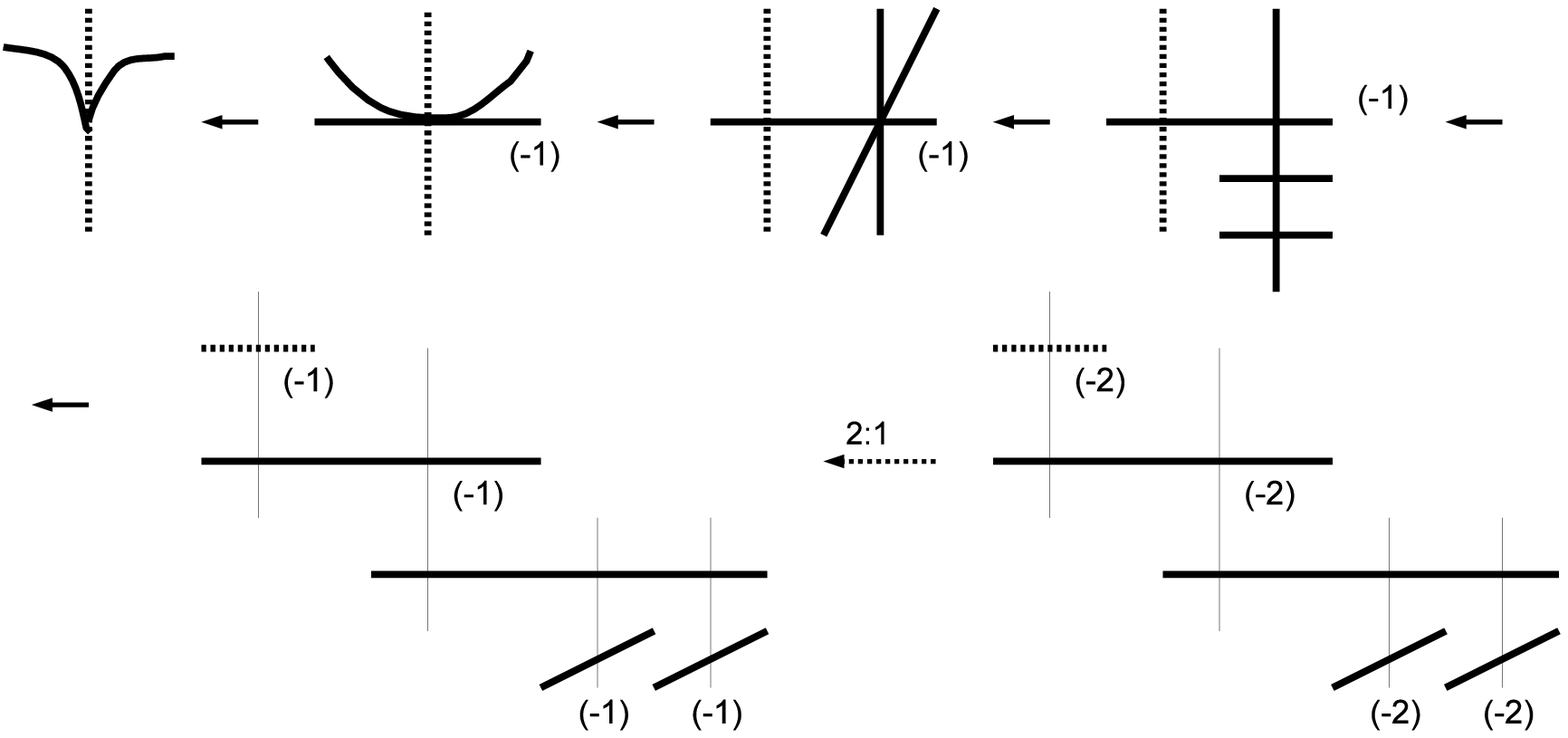}}
\end{center}
\end{proof}
\vspace{0.2cm}

\begin{teo}\label{evengenus}
Let $n\geq 2$ be an even integer. There exists a fibration of genus $g$ $f\colon S\longrightarrow \puno$ 
such that \\
$(i)$ $f$ is of general type;\\
$(ii)$ the surface $S$ is of general type and simply connected;\\
$(iii)$ $f$ admits fibres of type $(A)$ and no multiple fibres.
\end{teo}
\begin{proof}
Let us use the fibrations $f^n\colon Z^n\longrightarrow \puno$ constructed in Proposition \ref{echecacchio}.
Consider a cyclic base change $\beta\colon \puno\longrightarrow \puno$ of degree $d\geq 6$ ramifying in two points different from the points corresponding to singular fibres of $f^n$, and form the fibred product 
\begin{equation}\label{diag}
\xymatrix{
X^n \ar[r] \ar[d]_{g^n}& Z^n\ar[d]^{f^n}\\
\puno \ar[r]^\beta & \puno}
\end{equation}
So, $X^n$ has at least $6$ fibres of the form described in Proposition \ref{echecacchio}. By  the very same arguments as in Proposition \ref{cornalba}, $X^n$ is simply connected. Note that all its fibres are fibres also of $f^n$. If $g^n$ had a multiple fibre, then it would have at least $6$ copies of this fibre. From Remark \ref{multiple} we see that this would imply that $X^n$ is not simply connected, a contradiction.
\end{proof}

\medskip

\flushright {\em Dipartimento di Fisica e Matematica, \\via Valleggio 11, 22100 Como (Italy)}
\flushright{\em Email: lidia.stoppino@uninsubria.it}

\addcontentsline{toc}{section}{References}

\begin{thebibliography}{9}
\vspace{-0.1cm}
\bibitem{Abr} D. Abramovich, {\em Birational geometry for number theorists}, {Arithmetic geometry}, {Clay Math. Proc.}, {\bf 8}, {335--373}, {Amer. Math. Soc.}, 2009.
\vspace{-0.1cm}
\bibitem{BHPVdV} W. P. Barth, K. Hulek, C. A. M. Peters, A. Van de Ven, \emph{Compact complex surfaces}, second edition, 
Springer-Verlag, Berlin Heidelberg, 2004.
\vspace{-0.1cm}
\bibitem{Beaumin} A. Beauville, {\em Le nombre minimum de fibres singulieres d'une courbe stable sur $\puno$}, Asterisque {\bf 86} (1981), 97-108.
\vspace{-0.1cm}
\bibitem{Cfm} F. Campana, {\em Fibres multiples sur les surfaces: aspects geom\'etriques, hyperboliques et arithmetiques}, Manuscripta Math. {\bf 117} (2005) 429-461.
\vspace{-0.1cm}
\bibitem{CamN} F. Campana, {\em Negativity of compact curves in infinite covers of projective surfaces}, J. Algebraic Geom., {\bf 7}(4), (1998), 673-693.
 \vspace{-0.1cm}     
\bibitem{Cam} F. Campana, {\em Orbifolds, special varieties and classification theory}, Ann. Inst. Fourier (Grenoble), {\bf 54}(3) (2004), 499-630.
\vspace{-0.1cm}
\bibitem{camsurvey} F. Campana, {\em Special orbifolds and birational classification: a survey}, arXiv:1001.3763v1 [math.AG].
\vspace{-0.1cm}
\bibitem{CT/S/SD} J.-L. Colliot-Th{\'e}l{\`e}ne, A. N. Skorobogatov, P. Swinnerton-Dyer, {\em Double fibres and double covers: paucity of rational points}, Acta Arith. {\bf 79}(2) (1997), 113-135.
\vspace{-0.1cm}
\bibitem{cornish} M.D.T. Cornalba, {\em A remark on the topology of cyclic coverings of algebraic varieties (in italian)},  {Boll. Un. Mat. Ital. A (5)}, {\bf 18}(2) (1981), {323-328}.
\vspace{-0.1cm}
\bibitem{H-M} J. Harris, I. Morrison, \emph{Moduli of curves.} Springer-Verlag, New York, 1998.
\vspace{-0.1cm}
\bibitem{Har} R. Hartshorne, {\em Algebraic geometry}, Graduate Texts in Mathematics, No. 52, Springer-Verlag, New York, 1977.
\vspace{-0.1cm}
\bibitem{namba} M.  Namba,  {\em Branched coverings and algebraic functions},  Pitman Research Notes in Mathematics Series,
{\bf 161}, Longman Scientific \& Technical, Harlow, 1987.
\vspace{-0.1cm}
\bibitem{N-U} Y. Namikawa, K. Ueno, {\em The complete classification of fibres in pencils of curves of genus two}, Manuscripta Math. {\bf 9} (1973), 143-186.
\vspace{-0.1cm}
\bibitem{ogg2} A. P. Ogg, {\em On pencils of curves of genus two}, Topology, {\bf 5} (1966), 355-362.
\vspace{-0.1cm}
\bibitem{P} U. Persson, \emph{Chern invariants of surfaces of general type}, Composito Math. {\bf 43} (1982), 3-58.
\vspace{-0.1cm}
\bibitem{Tanmin2} S.L. Tan, Y. Tu, A. G. Zamora, {\em On complex surfaces with 5 or 6 semistable singular fibers over {$\puno$}}, Math Z. {\bf 249} (2005), 427-438.
\bibitem{Win} G. B. Winters, {\em On the existence of certain families of curves}, American Journal of Mathematics, {\bf 96}, No. 2 (1974), 215-228.
\vspace{-0.1cm}
\bibitem{XLNM} G. Xiao, {\em Surfaces fibr\'ees en courbes de genre deux}, {LNM}, {vol. 1137}, {Springer-Verlag},  {Berlin}, {1985}.
\end{thebibliography}
\end{document}